\UseRawInputEncoding

\documentclass[oneside, 11pt]{amsart} 

\usepackage{amsmath,amsthm,amssymb,epic}
\usepackage{eqlist,eqparbox}
\usepackage{amsfonts}
\usepackage{latexsym}
\usepackage[2emode]{psfrag}
\usepackage{amsthm}
\usepackage{amsmath}
\usepackage[all]{xy}

\newcommand{\Title}[1]{\bigskip\bigskip\centerline{\bf #1}\bigskip}
\newcommand{\Author}[1]{\medskip\centerline{ \it #1}}

\newcommand{\Affiliation}[1]{\medskip\centerline{#1}}
\newcommand{\Email}[1]{\medskip\centerline{#1}\bigskip}

\begin{document}

\newcommand{\N}{\mbox {$\mathbb N $}}
\newcommand{\Z}{\mbox {$\mathbb Z $}}
\newcommand{\Q}{\mbox {$\mathbb Q $}}
\newcommand{\R}{\mbox {$\mathbb R $}}
\newcommand{\lo }{\longrightarrow }
\newcommand{\ul}{\underleftarrow }
\newcommand{\rl}{\underrightarrow }
\newcommand{\rs }{\rightsquigarrow }
\newcommand{\ra }{\rightarrow } 
\newcommand{\dd }{\rightsquigarrow } 
\newcommand{\rars }{\Leftrightarrow }
\newcommand{\ol }{\overline }
\newcommand{\la }{\langle }
\newcommand{\tr }{\triangle }
\newcommand{\xr }{\xrightarrow }
\newcommand{\de }{\delta }
\newcommand{\pa }{\partial }
\newcommand{\LR }{\Longleftrightarrow }
\newcommand{\Ri }{\Rightarrow }
\newcommand{\va }{\varphi }
\newcommand{\Den}{{\rm Den}\,}
\newcommand{\Ker}{{\rm Ker}\,}
\newcommand{\Reg}{{\rm Reg}\,}
\newcommand{\Fix}{{\rm Fix}\,}
\newcommand{\Sup}{{\rm Sup}\,}
\newcommand{\Inf}{{\rm Inf}\,}
\newcommand{\Img}{{\rm Im}\,}
\newcommand{\Id}{{\rm Id}\,}
\newcommand{\ord}{{\rm ord}\,}

\newtheorem{theorem}{Theorem}[section]
\newtheorem{lemma}[theorem]{Lemma}
\newtheorem{proposition}[theorem]{Proposition}
\newtheorem{corollary}[theorem]{Corollary}
\newtheorem{definition}[theorem]{Definition}
\newtheorem{example}[theorem]{Example}
\newtheorem{examples}[theorem]{Examples}
\newtheorem{xca}[theorem]{Exercise}
\theoremstyle{remark}
\newtheorem{remark}[theorem]{Remark}
\newtheorem{remarks}[theorem]{Remarks}
\numberwithin{equation}{section}

\def\leftmark{L.C. Ciungu}

\Title{FOULIS-HOLLAND THEOREM FOR IMPLICATIVE-ORTHOMODULAR LATTICES} 
\title[Foulis-Holland theorem for implicative-orthomodular lattices]{}
                                                                       
\Author{\textbf{LAVINIA CORINA CIUNGU}}
\Affiliation{Department of Mathematics} 
\Affiliation{St Francis College}
\Affiliation{179 Livingston Street, Brooklyn, NY 11201, USA}
\Email{lciungu@sfc.edu}

\begin{abstract} 
We introduce the notion of distributivity for implicative-orthomodular lattices, proving an analogue result of the  
Foulis-Holland theorem. Based on this result, we characterize the distributive implicative-orthomodular lattices.
Moreover, we define the center of an implicative-orthomodular lattice as the set of all elements that commute with 
all other elements, and we prove that the center is an implicative-Boolean algebra. 
Additionally, we give new characterizations of implicative-orthomodular lattices. \\

\noindent
\textbf{Keywords:} {implicative BE algebra, implicative-orthomodular lattice, distributivity, Foulis-Holland theorem, implicative-Boolean center, commutor} \\
\textbf{AMS classification (2020):} 06C15, 03G25, 06A06, 81P10
\end{abstract}

\maketitle

\section{Introduction} 

Orthomodular lattices were first introduced by G. Birkhoff and J. von Neumann, by studying the structure of the 
lattice of projection operators on a Hilbert space.  
They introduced the orthomodular lattices as the event structure describing quantum mechanical experiments \cite{Birk1}. Later, orthomodular lattices and orthomodular posets were considered as the standard quantum logic \cite{Varad1} (see also \cite{Beran, DvPu, DvKu}). 
Complete studies on orthomodular lattices are presented in the monographs \cite{Birk2} and \cite{Beran}, while   
details on distributive lattices can be found in \cite{Balbes}. 
Based on the commutativity relation, D.J. Foulis (\cite{Foulis1}) and, independently, S.S. Holland (\cite{Holl1}) 
proved that, given a triple of elements in an orthomodular lattice, if at least one of them commutes with other two, 
then the triple is distributive.  
We also mention that certain results of continuous geometry were presented in \cite{Holl2} for the case of complete orthomodular lattices. \\
An orthomodular lattice (OML for short) is an ortholattice $(X,\wedge,\vee,^{'},0,1)$ verifying \\
$(OM)$ $(x\wedge y)\vee ((x\wedge y)^{'}\wedge x)=x$, for all $x,y\in X$, or, equivalently \\ 
$(OM^{'})$ $x\vee (x^{'}\wedge y)=y$, whenever $x\le y$ (where $x\le y$ iff $x=x\wedge y$) (\cite{PadRud}). \\
The orthomodular lattices were redefined and studied by A. Iorgulescu (\cite{Ior32}) based on involutive m-BE algebras, 
and two generalizations of orthomodular lattices were given: orthomodular softlattices and orthomodular widelattices. 
Based on implicative involutive BE algebras, we redefined in \cite{Ciu83} the orthomodular lattices, by introducing 
and studying the implicative-orthomodular lattices (IOML for short), giving certain characterizations of these structures. \\
In this paper, we define the notion of distributivity for implicative-orthomodular lattices and we investigate for  these structures certain results presented in \cite{Holl1} in the case of orthomodular lattices. 
We prove the Foulis-Holland theorem for implicative-orthomodular lattices, and, based on this result, we characterize 
the distributive implicative-orthomodular lattices. 
We also show that an implicative involutive BE algebra is an implicative-orthomodular lattice if and only if the 
commutativity relation is symmetric. 
We introduce the notion of the center $\mathcal{C}(X)$ of an implicative-orthomodular lattice $X$ as the 
set of all elements of $X$ that commute with all other elements of $X$, and we prove that $\mathcal{C}(X)$ is an implicative-Boolean subalgebra of $X$. 
We characterize $\mathcal{C}(X)$ with respect to the set $\mathcal{K}(x)$ of all complements of the 
elements $x\in X$. 
We also define the commutor of an implicative-orthomodular lattice $X$ as the generalization of the concept 
of the center of $X$, and we prove that the commutor is a sublattice of $X$ containing $\mathcal{C}(X)$. 
Finnaly, we apply the Foulis-Holland theorem to prove new properties and characterizations of 
implicative-orthomodular lattices.

$\vspace*{1mm}$

\section{Preliminaries}

We recall some basic notions and results regarding BE algebras that will be used in the paper. 

\emph{BE algebras} were introduced in \cite{Kim1} as algebras $(X,\ra,1)$ of type $(2,0)$ satisfying the 
following conditions, for all $x,y,z\in X$: 
$(BE_1)$ $x\ra x=1;$ 
$(BE_2)$ $x\ra 1=1;$ 
$(BE_3)$ $1\ra x=x;$ 
$(BE_4)$ $x\ra (y\ra z)=y\ra (x\ra z)$. 
A relation $\le$ is defined on $X$ by $x\le y$ iff $x\ra y=1$. 
A BE algebra $X$ is \emph{bounded} if there exists $0\in X$ such that $0\le x$, for all $x\in X$. 
In a bounded BE algebra $(X,\ra,0,1)$ we define $x^*=x\ra 0$, for all $x\in X$. 
A bounded BE algebra $X$ is called \emph{involutive} if $x^{**}=x$, for any $x\in X$. 

\begin{lemma} \label{qbe-10} $\rm($\cite{Ciu83}$\rm)$ 
Let $(X,\ra,1)$ be a BE algebra. The following hold for all $x,y,z\in X$: \\
$(1)$ $x\ra (y\ra x)=1;$ 
$(2)$ $x\le (x\ra y)\ra y$. \\
If $X$ is bounded, then: \\
$(3)$ $x\ra y^*=y\ra x^*;$ 
$(4)$ $x\le x^{**}$. \\
If $X$ is involutive, then: \\
$(5)$ $x^*\ra y=y^*\ra x;$ 
$(6)$ $x^*\ra y^*=y\ra x;$ 
$(7)$ $(x\ra y)^*\ra z=x\ra (y^*\ra z);$ \\
$(8)$ $x\ra (y\ra z)=(x\ra y^*)^*\ra z;$    
$(9)$ $(x^*\ra y)^*\ra (x^*\ra y)=(x^*\ra x)^*\ra (y^*\ra y)$.  
\end{lemma}

\noindent
In a BE algebra $X$, we define the additional operation: \\
$\hspace*{3cm}$ $x\Cup y=(x\ra y)\ra y$. \\
If $X$ is involutive, we define the operation: \\
$\hspace*{3cm}$ $x\Cap y=((x^*\ra y^*)\ra y^*)^*$, \\
and the relations $\le_Q$ and $\le_L$ by: \\
$\hspace*{3cm}$ $x\le_Q y$ iff $x=x\Cap y$ and $x\le_L y$ iff $x=(x\ra y^*)^*$. \\
\noindent
We can see that $x\le_L y$ iff $x^*=y\ra x^*$. 

\begin{lemma} \label{qbe-10-10}Let $(X,\ra,1)$ be an involutive BE algebra. 
The following holds \\
$\hspace*{2cm}$ $(x_1\ra y_1^*)^*\ra (x_2\ra y_2^*)=(x_1\ra x_2^*)^*\ra (y_1\ra y_2^*)$, \\ 
for all $x_1,x_2,y_1,y_2\in X$. 
\end{lemma}
\begin{proof}
Using twice Lemma \ref{qbe-10}$(7)$, we get:\\
$\hspace*{1.00cm}$ $(x_1\ra y_1^*)^*\ra (x_2\ra y_2^*)=x_1\ra (y_1\ra (x_2\ra y_2^*))$ \\
$\hspace*{5.20cm}$ $=x_1\ra (x_2\ra (y_1\ra y_2^*))$ \\
$\hspace*{5.20cm}$ $=(x_1\ra x_2^*)^*\ra (y_1\ra y_2^*)$. 
\end{proof}

\begin{proposition} \label{qbe-20} $\rm($\cite{Ciu83}$\rm)$ Let $X$ be an involutive BE algebra. 
Then the following hold for all $x,y,z\in X$: \\
$(1)$ $x\le_Q y$ implies $x=y\Cap x$ and $y=x\Cup y;$ \\
$(2)$ $\le_Q$ is reflexive and antisymmetric; \\
$(3)$ $x\Cap y=(x^*\Cup y^*)^*$ and $x\Cup y=(x^*\Cap y^*)^*;$ \\ 
$(4)$ $x\le_Q y$ implies $x\le y;$ \\
$(5)$ $x, y\le_Q z$ and $z\ra x=z\ra y$ imply $x=y$ \emph{(cancellation law)}; \\
$(6)$ $x\ra ((y\ra x^*)^*\Cup z)=y\Cup (x\ra z)$.    
\end{proposition}

\begin{lemma} \label{qbe-20-10} Let $X$ be an involutive BE algebra. Then: \\ 
$(1)$ $\le$ is not an order relation; \\
$(2)$ $\le_L$ implies $\le$; \\
$(3)$ $\le_L$ is an order relation on $X$; \\
$(4)$ $z\le_L x$ and $z\le_L y$ imply $z\le_L (x\ra y^*)^*$.   
\end{lemma}
\begin{proof}
$(1)$ $\le$ is not transitive. \\
$(2)$ If $x\le_L y$, then $x=(x\ra y^*)^*$, so that \\
$\hspace*{2.00cm}$ $x\ra y=(x\ra y^*)^*\ra y=y^*\ra (x\ra y^*)=x\ra (y^*\ra y^*)=1$, \\
hence $x\le y$. \\
$(3)$ If $x\le_L y$ and $y\le_L x$, then $x^*=x\ra y^*=y^*$, hence $x=y$. 
From $x\le_L y$ and $y\le_L z$, we get: \\
$\hspace*{2.00cm}$ $x^*=x\ra y^*=x\ra (z\ra y^*)=z\ra (x\ra y^*)=z\ra x^*$, \\
hence $x\le_L z$. Since, obviously, $\le_L$ is reflexive and antisymmetric, it follows that it is an order relation. \\
$(4)$ From $z\le_L x$ and $z\le_L y$ we have $z^*=x\ra z^*$ and $z^*=y\ra z^*$. 
It follows that $z^*=x\ra (y\ra z^*)$, and by Lemma \ref{qbe-10}$(7)$, we get $z^*=(x\ra y^*)^*\ra z^*$, 
that is $z\le_L (x\ra y^*)^*$.  
\end{proof}

\noindent 
A \emph{quantum-Wajsberg algebra} (\emph{QW algebra, for short}) (\cite{Ciu83}) $(X,\ra,^*,1)$ is an 
involutive BE algebra $(X,\ra,^*,1)$ satisfying the following condition: for all $x,y,z\in X$, \\
(QW) $x\ra ((x\Cap y)\Cap (z\Cap x))=(x\ra y)\Cap (x\ra z)$. \\
It was proved in \cite{Ciu83} that condition (QW) is equivalent to the following conditions: \\
$(QW_1)$ $x\ra (x\Cap y)=x\ra y;$ \\ 
$(QW_2)$ $x\ra (y\Cap (z\Cap x))=(x\ra y)\Cap (x\ra z)$. \\
An involutive-BE algebra satisfying condition $(QW_2)$ is called an \emph{implicative-orthomolular algebra}. 

$\vspace*{1mm}$

\section{On implicative-orthomodular lattices}

In this section, we recall certain notions and results regarding the implicative-orthomodular lattices (\cite{Ciu83}), and we prove new properties and characterizations of these structures. 

\begin{definition} \label{iol-20} 
\emph{
A BE algebra is called \emph{implicative} if it satisfies the following condition: for all $x,y\in X$, \\
$(Impl)$ $(x\ra y)\ra x=x$. 
}
\end{definition}

\begin{lemma} \label{iol-30} \emph{(\cite{Ciu83})} Let $(X,\ra,^*,1)$ be an implicative involutive BE algebra. 
Then $X$ verifies the following axioms: for all $x,y\in X$, \\
$(iG)$       $x^*\ra x=x$, or equivalently, $x\ra x^*=x^*;$ \\
$(Iabs$-$i)$ $(x\ra (x\ra y))\ra x=x;$ \\
$(Pimpl)$    $x\ra (x\ra y)=x\ra y$.    
\end{lemma}

\begin{lemma} \label{iol-30-10} \emph{(\cite{Ciu83})}
Let $(X,\ra,^*,1)$ be an involutive BE algebra. The following are equivalent: \\
$(a)$ $X$ is implicative; \\
$(b)$ $X$ verifies axioms $(iG)$ and $(Iabs$-$i)$; \\ 
$(c)$ $X$ verifies axioms $(Pimpl)$ and $(Iabs$-$i)$. 
\end{lemma}

\begin{lemma} \label{iol-30-20} Let $X$ be an implicative involutive BE algebra. The following hold, 
for all $x,y\in X$: \\ 
$(1)$ $x\le_L y$ iff $y^*\le_L x^*;$ \\
$(2)$ $\le_Q$ implies $\le_L$. 
\end{lemma} 
\begin{proof}
$(1)$ If $x\le_L y$, then $x^*=x\ra y^*=y\ra x^*$, and we get $y^*\ra (x^*)^*=y^*\ra x=  
 x^*\ra y=(y\ra x^*)\ra y=y$ (by $(Impl)$), hence $y^*=(y^*\ra (x^*)^*)^*$, so that $y^*\le_L x^*$. 
Conversely, if $y^*\le_L x^*$, then $y=x^*\ra y$ and we get $x\ra y^*=y\ra x^*=(x^*\ra y)\ra x^*=x^*$ (by $(Impl)$), 
so that $x\le_L y$. \\
$(2)$ Assume that $x\le_Q y$, that is $x\Cap y=x$. Since by Lemma \ref{iol-30}, $(Impl)$ implies $(iG)$, we have 
$y\ra y^*=y^*$. It follows that: \\
$\hspace*{2.00cm}$ $x\ra y^*=y\ra x^*=y\ra (x\Cap y)^*=y\ra ((x^*\ra y^*)\ra y^*)$ \\
$\hspace*{3.35cm}$ $=(x^*\ra y^*)\ra (y\ra y^*)=(x^*\ra y^*)\ra y^*=(x\Cap y)^*=x^*$. \\
It follows that $x=(x\ra y^*)^*$, hence $x\le_L y$. 
\end{proof}

\begin{lemma} \label{ioml-30} Let $X$ be an involutive BE algebra. 
The following are equivalent for all $x,y\in X$: \\
$(IOM)$ $x\Cap (y\ra x)=x;$ \\
$(IOM^{'})$ $x\Cap (x^*\ra y)=x;$ $\hspace*{5cm}$ \\
$(IOM^{''})$ $x\Cup (x\ra y)^*=x$. 
\end{lemma}
\begin{proof}
The proof is straightforward. 
\end{proof}

\begin{definition} \label{ioml-30-10} \emph{
An \emph{implicative-orthomodular lattice} (IOML for short) is an implicative involutive BE algebra satisfying 
any one (and hence all) of the equivalent conditions of Lemma \ref{ioml-30}
}
\end{definition}

The orthomodular lattices $(X,\wedge,\vee,^{'},0,1)$ are defitionally equivalent to implicative-orthomodular 
lattices $(X,\ra,^*,1)$, by the mutually inverse transformations \\
$\hspace*{3cm}$ $\varphi:$\hspace*{0.2cm}$ x\ra y:=(x\wedge y^{'})^{'}$ $\hspace*{0.1cm}$ and  
                $\hspace*{0.1cm}$ $\psi:$\hspace*{0.2cm}$ x\wedge y:=(x\ra y^*)^*$, \\
and the relation $x\vee y:=(x^{'}\wedge y^{'})^{'}=x^*\ra y$. 

\begin{proposition} \label{ioml-40} \emph{(\cite{Ciu83})}
Let $X$ be an implicative-orthomodular lattice. Then the following hold for all $x,y,z\in X$: \\
$(1)$ $x\Cap (y\Cup x)=x$ and $x\Cup (y\Cap x)=x$. \\
If $x\le_Q y$, then: \\
$(2)$ $y\Cup x=y$ and $y^*\le_Q x^*;$ \\ 
$(3$ $y\ra z\le_Q x\ra z$ and $z\ra x\le_Q z\ra y;$ \\
$(4)$ $x\Cap z\le_Q y\Cap z$ and $x\Cup z\le_Q y\Cup z$. 
\end{proposition}

\begin{proposition} \label{ioml-50} \emph{(\cite{Ciu83})}
Let $X$ be an implicative-orthomodular lattice. Then the following hold for all $x,y,z\in X$: \\
$(1)$ $x\ra (y\Cap x)=x\ra y;$ \\  
$(2)$ $(x\Cup y)\ra (x\ra y)^*=y^*;$ \\ 
$(3)$ $x\Cap ((y\ra x)\Cap (z\ra x))=x;$ \\
$(4)$ $(x\ra y)\ra (y\Cap x)=x;$ \\
$(5)$ $\le_Q$ is an order relation on $X;$ \\
$(6)$ $x\le_Q y$ and $x\le y$ imply $x=y;$ \\
$(7)$ $x\Cap y\le_Q y\le_Q x\Cup y$.       
\end{proposition}

\begin{lemma} \label{ioml-60-10} Let $X$ be an implicative-orthomodular lattice. Then the following hold for all $x,y,z,x_1,x_2,y_1,y_2\in X$: \\
$(1)$ $x\le_Q y$ implies $(x\ra z^*)^*\le_Q (y\ra z^*)^*;$ \\
$(2)$ $y\le_Q x$ and $z\le_Q x$ imply $y^*\ra z\le_Q x;$ \\
$(3)$ $x_1\le_Q y_1$, $x_2\le_Q y_2$ imply $x_1^*\ra x_2\le_Q y_1^*\ra y_2;$ \\
$(4)$ $(x\ra y^*)\ra (x\ra y)^*\le_Q x;$ \\
$(5)$ $(z\ra (z\ra x^*)^*)^*\le_Q z^*\ra x$.  
\end{lemma}
\begin{proof}
$(1)$ $x\le_Q y$ implies $y\ra z^*\le_Q x\ra z^*$, hence $(x\ra z^*)^*\le_Q (y\ra z^*)^*$. \\
$(2)$ By hypothesis and $(iG)$, we get: $y^*\ge_Q x^*$, so that $y^*\ra z\le_Q x^*\ra z\le_Q x^*\ra x=x$. \\
$(3)$ By hypothesis, we have: $x_1^*\ge_Q y_1^*$, hence $x_1^*\ra x_2\le_Q y_1^*\ra x_2\le_Q y_1^*\ra y_2$. \\ 
$(4)$ By $(IOM^{'})$, $(x\ra y^*)^*\le_Q x$ and $(x\ra y)^*\le_Q x$, and applying $(2)$, it follows that 
$(x\ra y^*)\ra (x\ra y)^*\le_Q x$. \\
$(5)$ We have $x^*\Cap z\le_Q z\le_Q z^*\ra x$. 
Since $x^*\Cap z=(x\Cup z^*)^*=((x\ra z^*)\ra z^*)^*=(z\ra (z\ra x^*)^*)^*$, we get 
$(z\ra (z\ra x^*)^*)^*\le_Q z\le_Q z^*\ra x$.
\end{proof}

\begin{theorem} \label{ioml-90-10} An implicative involutive BE algebra $X$. The following are equivalent: \\
$(a)$ $X$ is an implicative-orthomodular lattice; \\
$(b)$ $X$ satisfies condition $(QW_1);$ \\ 
$(c)$ $X$ satisfies condition $(QW_2);$ \\ 
$(d)$ $X$ satisfies condition $(QW)$. 
\end{theorem}
\begin{proof} 
$(a)\Rightarrow (c)$ Assume that $X$ is an implicative-orthomodular lattice. 
By Proposition \ref{ioml-50}$(3)$, we have $x^*\Cup ((y\ra x)^*\Cup (z\ra x)^*))=x^*$. 
Replacing $x, y, z$ with $x^*, y^*, z^*$, respectively, we get $x\Cup ((y^*\ra x^*)^*\Cup (z^*\ra x^*))=x$. 
It follows that: \\
$\hspace*{1.00cm}$ $x\ra (y\Cap (z\Cap x))=x\ra (y\Cap (z^*\Cup x^*)^*)$ \\
$\hspace*{4.00cm}$ $=x\ra (y^*\Cup (z^*\Cup x^*))^*$ \\ 
$\hspace*{4.00cm}$ $=x\ra (y^*\Cup ((z^*\ra x^*)\ra x^*))^*$ \\
$\hspace*{4.00cm}$ $=x\ra (y^*\Cup (x\ra (z^*\ra x^*)^*))^*$ \\
$\hspace*{4.00cm}$ $=x\ra (x\ra ((y^*\ra x^*)^*\Cup (z^*\ra x^*)^*))^*$ (Prop. \ref{qbe-20}$(6)$) \\
$\hspace*{4.00cm}$ $=(x\Cup ((y^*\ra x^*)^*\Cup (z^*\ra x^*)^*))\ra (x\ra ((y^*\ra x^*)^*\Cup (z^*\ra x^*)^*))^*$ \
$\hspace*{4.00cm}$ $=((y^*\ra x^*)^*\Cup (z^*\ra x^*)^*)^*$ (Prop. \ref{ioml-50}$(2)$) \\
$\hspace*{4.00cm}$ $=(y^*\ra x^*)\Cap (z^*\ra x^*)=(x\ra y)\Cap (x\ra z)$. \\
Hence axiom $(QW_2)$ is satisfied. \\
$(c)\Rightarrow (a)$ Replacing $y:=0$ in $(QW_2)$, we have $x^*=x^*\Cap (x\ra z)$, and replacing $x$ with $x^*$ and $z$ with $y$, we get $x=x\Cap (x^*\ra y)$. It follows that $X$ satisfies condition $(IOM^{'}$), hence it is 
an implicative-orthomodular lattice. \\
$(c)\Rightarrow (b)$ and $(c)\Rightarrow (d)$ According to \cite{Ciu83}, if $X$ is an implicative involutive BE algebra, then axioms $(QW_1)$, $(QW_2)$, $(QW)$ are equivalent.  
\end{proof}

\begin{corollary} \label{ioml-90-20} Any implicative-orthomodular lattice is a quantum-Wajsberg algebra. 
\end{corollary}

\begin{proposition} \label{ioml-70} Let $X$ be an implicative-orthomodular lattice. 
The following hold for all $x,y\in X$: \\ 
$(1)$ $(x\Cap y)\ra (y\Cap x)=1;$ \\
$(2)$ $(x\Cup y)\ra (y\Cup x)=1$.  
\end{proposition}
\begin{proof}
According to Theorem \ref{ioml-90-10}, $X$ satisfies condition $(QW_1)$, and we have: \\
$\hspace*{1.00cm}$ $(x\Cap y)\ra (y\Cap x)=(y\Cap x)^*\ra (x\Cap y)^*=(y\Cap x)^*\ra ((x^*\ra y^*)\ra y^*)$ \\ 
$\hspace*{4.00cm}$ $=(x^*\ra y^*)\ra ((y\Cap x)^*\ra y^*)=(y\ra x)\ra (y\ra (y\Cap x))$ \\
$\hspace*{4.00cm}$ $=(y\ra x)\ra (y\ra x)=1$. \\
$(2)$ Using $(1)$, we get: \\
$\hspace*{1.00cm}$ $(x\Cup y)\ra (y\Cup x)=(x^*\Cap y^*)^*\ra (y^*\Cap x^*)^*=(y^*\Cap x^*)\ra (x^*\Cap y^*)=1$.    
\end{proof}

\begin{theorem} \label{dioml-05-50} Let $X$ be an implicative involutive BE algebra. 
The following are equivalent: \\
$(a)$ $X$ is an implicative-orthomodular lattice; \\
$(b)$ $\le_L$ implies $\le_Q;$ \\
$(c)$ for $x,y\in X$, $x\le_L y$ implies $y=y\Cup x$.  
\end{theorem}
\begin{proof}
$(a) \Rightarrow (b)$ Let $x,y\in X$ such that $x\le_L y$, that is $y\ra x^*=x^*$. 
Then we get: \\
$\hspace*{2.00cm}$ $x\Cap y=((x^*\ra y^*)\ra y^*)^*=(((y\ra x^*)\ra y^*)\ra y^*)^*$ \\
$\hspace*{2.95cm}$ $=(((x\ra y^*)\ra y^*)\ra y^*)^*=((x\Cup y^*)\ra y^*)^*$ \\
$\hspace*{2.95cm}$ $=(y\ra (x\Cup y^*)^*)^*=(y\ra (x^*\Cap y))^*=(y\ra x^*)^*$  (Prop. \ref{ioml-50}$(1)$) \\ 
$\hspace*{2.95cm}$ $=(x\ra y^*)^*=x$. \\
Thus $x=x\Cap y$, that is $x\le_Q y$. \\ 
$(b) \Rightarrow (a)$ Let $x,y\in X$ such that $x\le_L y$ implies $x\le_Q y$. \\
Since by $(Impl)$, $(x\ra (y\ra x)^*)^*=((y\ra x)\ra x^*)^*=((x^*\ra y^*)\ra x^*)^*=(x^*)^*=x$, it follows that 
$x\le_L y\ra x$. Hence $x\le_Q y\ra x$, that is $x\Cap (y\ra x)=x$, so that $X$ verifies condition $(IOM)$. 
Thus $X$ is an implicative-orthomodular lattice. \\  
$(b) \Rightarrow (c)$ Let $x,y\in X$ such that $x\le_L y$. It follows that $y^*\le_L x^*$, and by $(b)$, 
$y^*\le_Q x^*$, so that $y^*=y^*\Cap x^*$, that is $y=y\Cup x$. \\
$(c) \Rightarrow (b)$ If $x\le_L y$, then $y^*\le_L x^*$, so that, by $(c)$, $x^*=x^*\Cup y^*$. 
Hence $x=x\Cap y$, that is $x\le_Q y$. 
\end{proof}

\begin{corollary} \label{dioml-10-20} If $X$ is an implicative-orthomodular lattice, then $\le_Q=\le_L$. 
\end{corollary}
\begin{proof}
It follows by Theorem \ref{dioml-05-50} and Lemma \ref{iol-30-20}. 
\end{proof}

$\vspace*{1mm}$

\section{Foulis-Holland theorem for implicative-orthomodular lattices} 

In this section, we define and study the notion of distributivity for implicative-orthomodular lattices. 
Similarly as in the case of orthomodular lattices (\cite{Holl1}), we prove the Foulis-Holland theorem for implicative-orthomodular lattices. 
We show that an implicative-orthomodular lattice is distributive if and only if it satisfies the divisibility condition. 
Moreover, we prove that an implicative involutive BE algebra is an implicative-orthomodular lattice if and only if the 
commutativity relation is symmetric. 

\begin{definition} \label{dioml-05}
\emph{
Let $X$ be an implicative involutive BE algebra. We say that $X$ satisfies the \emph{divisibility condition} 
if, for all $x,y\in X$: \\
$(Idiv)$ $x\ra (x\ra y)^*=x\ra y^*$. 
}
\end{definition}

If $x,y\in X$ such that $x\ra (x\ra y)^*=x\ra y^*$, we write $(x,y)Idiv$. 
Obviously, if $(x,y)Idiv$ for all $x,y\in X$, then $X$ verifies condition $(Idiv)$. 

\begin{lemma} \label{dioml-05-10}
Let $X$ be an implicative involutive BE algebra and let $x,y\in X$. If $(x,y)Idiv$, then $x\le y$ implies $x\le_L y$. 
\end{lemma}
\begin{proof}  
Let $x,y\in X$ such that $x\le y$, that is $x\ra y=1$. Since $(x,y)Idiv$, we get: \\
$(x\ra y^*)^*=(x\ra (x\ra y)^*)^*=(x\ra 1^*)^*=x$, hence $x\le_L y$. 
\end{proof}

\begin{definition} \label{dioml-10} 
\emph{
Let $X$ be an implicative involutive BE algebra and let $x,y\in X$. 
We say that $x$ \emph{commutes} with $y$, denoted by $x\mathcal{C} y$, if $x=(x\ra y^*)\ra (x\ra y)^*$. 
}
\end{definition}

\begin{lemma} \label{dioml-20} 
Let $X$ be an implicative involutive BE algebra and let $x\in X$. Then: \\
$(1)$ $x\mathcal{C} x$, $x\mathcal{C} 0$, $0\mathcal{C} x$, $x\mathcal{C} 1$, $1\mathcal{C} x$, $x\mathcal{C} x^*$, 
$x^*\mathcal{C} x;$ \\ 
$(2)$ $x\le_L y$ or $x\le_L y^*$ implies $x\mathcal{C} y;$ \\
$(3)$ $x\le_Q y$ or $x\le_Q y^*$ implies $x\mathcal{C} y;$ \\
$(4)$ $x\mathcal{C} (y\ra x).$ 
\end{lemma}
\begin{proof} 
$(1)$ It is straightforward. \\
$(2)$ Since $x\le_L y$, by $(Impl)$,  
$(x\ra y^*)\ra (x\ra y)^*=x^*\ra (x\ra y)^*=(x\ra y)\ra x=x$, hence $x\mathcal{C} y$. 
Similarly, $x\le_L y^*$ implies $x^*=x\ra y$. 
Then, by $(Impl)$, we get $(x\ra y^*)\ra (x\ra y)^*=(x\ra y^*)\ra x=x$, that is $x\mathcal{C} y$. \\
$(3)$ It follows by $(2)$, since by Lemma \ref{iol-30-20}, $x\le_Q y$ implies $x\le_L y$. \\
$(4)$ It follows by $(2)$, since $x\le_L y\ra x$. 
\end{proof}

\begin{proposition} \label{dioml-10-05} Let $X$ be an implicative involutive BE algebra, and let $x,y\in X$. 
If $x\mathcal{C} y$, then $x\Cap y=(x\ra y^*)^*$. 
\end{proposition}
\begin{proof}
Suppose $x\mathcal{C} y$, that is $(x\ra y^*)\ra (x\ra y)^*=x$. It follows that: \\
$\hspace*{2.00cm}$ $(x\Cap y)^*=(x^*\ra y^*)\ra y^*=(y\ra x)\ra y^*$ \\
$\hspace*{3.50cm}$ $=(y\ra ((x\ra y^*)\ra (x\ra y)^*))\ra y^*$ \\
$\hspace*{3.50cm}$ $=((x\ra y^*)\ra (y\ra (x\ra y)^*))\ra y^*$ \\
$\hspace*{3.50cm}$ $=((x\ra y^*)\ra ((x\ra y)\ra y^*))\ra y^*$ \\
$\hspace*{3.50cm}$ $=((x\ra y^*)\ra ((y^*\ra x^*)\ra y^*))\ra y^*$ \\
$\hspace*{3.50cm}$ $=((x\ra y^*)\ra y^*)\ra y^*$ (by $(Impl)$) \\
$\hspace*{3.50cm}$ $=(x\Cup y^*)\ra y^*=y\ra (x^*\Cap y)=y\ra x^*$ (by Prop. \ref{ioml-50}$(1)$) \\
$\hspace*{3.50cm}$ $=x\ra y^*$. \\
Hence $x\Cap y=(x\ra y^*)^*$. 
\end{proof}

\begin{corollary} \label{dioml-10-05-10} Let $X$ be an implicative involutive BE algebra.  
If $x\mathcal{C} y$ for all $x,y\in X$, then $X$ is an implicative-orthomodular lattice. 
\end{corollary}
\begin{proof}
If $x\mathcal{C} y$, then by Proposition \ref{dioml-10-05}, $x\Cap y=(x\ra y^*)^*$, 
so that $x\le_L y$ if and only if $x\le_Q y$, for all $x,y\in X$. 
According to Theorem \ref{dioml-05-50}, $X$ is an implicative-orthomodular lattice. 
\end{proof}

\begin{theorem} \label{dioml-10-10} Let $X$ be an implicative involutive BE algebra. 
The following are equivalent: \\
$(a)$ $X$ is an implicative-orthomodular lattice; \\
$(b)$ whenever $x\mathcal{C} y$, then $y\mathcal{C} x$. 
\end{theorem}
\begin{proof}
$(a) \Rightarrow (b)$ Suppose $X$ is an implicative-orthomodular lattice, and let $x,y\in X$ such that 
$x\mathcal{C} y$, that is $x=(x\ra y^*)\ra (x\ra y)^*$. Then we have:\\
$\hspace*{1.30cm}$ $(y\ra x^*)\ra (y\ra x)^*=(y\ra x^*)\ra (y\ra ((x\ra y^*)\ra (x\ra y)^*))^*$ \\
$\hspace*{5.00cm}$ $=(y\ra x^*)\ra ((x\ra y^*)\ra (y\ra (x\ra y)^*))^*$ \\
$\hspace*{5.00cm}$ $=(y\ra x^*)\ra ((x\ra y^*)\ra ((x\ra y)\ra y^*))^*$ \\
$\hspace*{5.00cm}$ $=(y\ra x^*)\ra ((x\ra y^*)\ra ((y^*\ra x^*)\ra y^*))^*$ \\
$\hspace*{5.00cm}$ $=(y\ra x^*)\ra ((x\ra y^*)\ra y^*)^*$ (by $(Impl)$) \\
$\hspace*{5.00cm}$ $=(y\ra x^*)\ra (x\Cup y^*)^*$ \\
$\hspace*{5.00cm}$ $=(y\ra x^*)\ra (x^*\Cap y)$ \\
$\hspace*{5.00cm}$ $=y$ (by Prop. \ref{ioml-50}$(4)$). \\ 
Hence $y\mathcal{C} x$. \\
$(b) \Rightarrow (a)$ Let $x\le_L y$, so that $y^*\le_L x^*$. 
By Lemma \ref{dioml-20}, we have 
$y^*\mathcal{C} x^*$, and by $(b)$ we get $x^*\mathcal{C} y^*$, hence $x^*=(x^*\ra y)\ra (x^*\ra y^*)^*$.  
Since $y^*\le_L x^*$ implies $x^*\ra y=y$, we get: 
$x^*=y\ra (x^*\ra y^*)^*=(x^*\ra y^*)\ra y^*=(x^*\Cup y^*)$. Hence $x=x\Cap y$, that is $x\le_Q y$. 
Thus $\le_L$ implies $\le_Q $, so that, by Theorem \ref {dioml-05-50}, $X$ is an implicative-orthomodular lattice. 
\end{proof}

\begin{lemma} \label{dioml-30} 
Let $X$ be an implicative-orthomodular lattice and let $x,y\in X$. 
Then $x\mathcal{C} y$ implies $x\mathcal{C} y^*$, $x^*\mathcal{C} y$ and $x^*\mathcal{C} y^*$. 
\end{lemma}
\begin{proof}
Since $x\mathcal{C} y$, we have: $(x\ra (y^*)^*)\ra (x\ra y^*)^*=(x\ra y^*)\ra (x\ra y)^*=x$, thus $x\mathcal{C} y^*$. 
Using Theorem \ref{dioml-10-10}, it follows that $y^*\mathcal{C} x$, so that $y^*\mathcal{C} x^*$. 
Applying again Theorem \ref{dioml-10-10}, we get $x^*\mathcal{C} y^*$.  
\end{proof}

\begin{proposition} \label{dioml-60} 
Let $X$ be an implicative-orthomodular lattice, and let $x,y\in X$. Then $x\mathcal{C} y$ if and only if 
$(x,y)Idiv$. 
\end{proposition}
\begin{proof} 
Assume that $x\mathcal{C} y$. 
By Proposition \ref{ioml-40}$(2)$, $x\le_Q y$ implies $y=y\Cup x=x^*\ra (y\ra x)^*$. \\
Since $y\le_Q x\ra y$, we get $(x\ra y)^*\le_Q y^*$, so that $x\ra (x\ra y)^*\le_Q x\ra y^*$. 
Then, by above remark, we get: \\ 
$\hspace*{2.00cm}$ $x\ra y^*=(x\ra (x\ra y)^*)^*\ra ((x\ra y^*)\ra (x\ra (x\ra y)^*))^*$ \\
$\hspace*{3.25cm}$ $=(x\ra (x\ra y)^*)^*\ra (x\ra ((x\ra y^*)\ra (x\ra y)^*))^*$ \\
$\hspace*{3.25cm}$ $=(x\ra (x\ra y)^*)^*\ra (x\ra x)^*$ (since $x\mathcal{C} y$, $(x\ra y^*)\ra (x\ra y)^*=x$) \\
$\hspace*{3.25cm}$ $=(x\ra (x\ra y)^*)^*\ra 0=x\ra (x\ra y)^*$. \\
Hence $(x,y)Idiv$. 
Conversely, suppose $(x,y)Idiv$, that is $x\ra (x\ra y)^*=x\ra y^*$. 
Since by $(IOM^{''})$, $x\Cup (x\ra y)^*=x$, we get: \\
$\hspace*{2.00cm}$ $x=(x\ra (x\ra y)^*)\ra (x\ra y)^*=(x\ra y^*)\ra (x\ra y)^*$, \\
that is $x\mathcal{C} y$.  
\end{proof}

\begin{proposition} \label{dioml-70} 
Let $X$ be an implicative-orthomodular lattice, and let $x,y\in X$. Then $x\mathcal{C} y$ if and only if 
$y\Cap x=(x\ra y^*)^*$. 
\end{proposition}
\begin{proof} 
If $x\mathcal{C} y$, by Proposition \ref{dioml-60}, $(x,y)Idiv$, and we have: \\
$\hspace*{1.00cm}$ $(x\ra y^*)^*=(x\ra (x\ra y)^*)^*=((x\ra y)\ra x^*)^*=((y^*\ra x^*)\ra x^*)^*=y\Cap x$. \\
Conversely, assume that $y\Cap x=(x\ra y^*)^*$. Applying Proposition \ref{ioml-50}$(4)$, it follows that: \\
$\hspace*{1.00cm}$ $(x\ra y^*)\ra (x\ra y)^*=(y\Cap x)^*\ra (x\ra y)^*=(x\ra y)\ra (y\Cap x)=x$, \\
that is $x\mathcal{C} y$. 
\end{proof}

\begin{corollary} \label{dioml-70-10} 
Let $X$ be an implicative-orthomodular lattice, and let $x,y\in X$. Then $x\mathcal{C} y$ if and only if 
$x\Cap y=y\Cap x$. 
\end{corollary}
\begin{proof} 
By Proposition \ref{dioml-70}, $x\mathcal{C} y$ if and only if $y\Cap x=(x\ra y^*)^*$. 
Using Theorem \ref{dioml-10-10}, $x\mathcal{C} y$ if and only if $y\mathcal{C} x$ if and only if 
$x\Cap y=y\Cap x=(x\ra y^*)^*$.
\end{proof}

\begin{remark} \label{dioml-70-20} Let $X$ be an implicative-orthomodular lattice. \\
$(1)$ Corollary \ref{dioml-70-10} explains why relation $\mathcal{C}$ is called the 
\emph{commutativity relation} on $X$. \\
$(2)$ If $x\mathcal{C} y$ for all $x,y\in X$, then $X$ is commutative, so that it is a Wajsberg lattice. 
We will see in the next section that $X$ is an implicative-Boolean algebra. 
\end{remark}

\begin{definition} \label{dioml-80} 
\emph{
An implicative-orthomodular lattice $X$ is \emph{distributive} if it satisfies the following conditions, 
for all $x,y,z\in X$: \\
$(Idis_1)$ $((x^*\ra y)\ra z^*)^*=(x\ra z^*)\ra (y\ra z^*)^*$, \\
$(Idis_2)$ $((x\ra y^*)\ra z)^*=((z^*\ra x)\ra (z^*\ra y)^*$. 
}
\end{definition}

\begin{remark} \label{dioml-90} Using the mutually inverse transformations \\ 
$\hspace*{3cm}$ $\varphi:$\hspace*{0.2cm}$ x\ra y:=(x\wedge y^{'})^{'}$ $\hspace*{0.1cm}$ and  
                $\hspace*{0.1cm}$ $\psi:$\hspace*{0.2cm}$ x\wedge y:=(x\ra y^*)^*$, \\
we can easily see that $(Idis_1)$ and $(Idis_2)$ are equivalent to conditions: \\
$(D_1)$ $(x\vee y)\wedge z=(x\wedge z)\vee (y\wedge z)$ and \\
$(D_2)$ $(x\wedge y)\vee z=(x\vee z)\wedge (y\vee z)$, \\
from the definition of distributive lattices. 
\end{remark}

Similary to notations from \cite[Def. 3]{Holl1}, we write $(x,y,z)Idis_1$ and $(x,y,z)Idis_2$ if the triple $(x,y,z)$ verifies conditions $(Idis_1)$ and $(Idis_2)$, respectively. 
If $(x,y,z)Idis_1$ and $(x,y,z)Idis_2$, we write $(x,y,z)Idis$. 
Obviously, $(x,y,z)Idis_1$ if and only if $(x^*,y^*,z^*)Idis_2$. 
We say that $(x,y,z)$ is a distributive triple if $(x,y,z)Idis$, for all permutations of $\{x,y,z\}$. 

\begin{definition} \label{dioml-90-10} 
\emph{
Let $X$ be an implicative involutive BE algebra.  
We say that $X$ is \emph{distributive} or $X$ satisfies condition $(Idis)$ if all triples $(x,y,z)$ are distributive.  
}  
\end{definition}

\noindent
Now we can prove the Foulis-Holland theorem for implicative-orthomodular lattices. 

\begin{theorem} \label{dioml-100} \emph{(Foulis-Holland theorem)} 
Let $X$ be an implicative-orthomodular lattice and let $x,y,z\in X$, such that 
one of these elements commutes with the other two. Then $(x,y,z)$ is a distributive triple. 
\end{theorem}
\begin{proof}
Without loss of generality, suppose that $z\mathcal{C} x$ and $z\mathcal{C} y$. 
We need to verify six different relations: three for $Idis_1$ and three for $Idis_2$. 
From $z\mathcal{C} x$ and $z\mathcal{C} y$ we also have $z^*\mathcal{C} x^*$ and $z^*\mathcal{C} y^*$, so that 
$(x,y,z)Idis_1$ implies $(x^*,y^*,z^*)Idis_2$. 
Similary $(y,z,x)Idis_1$ implies $(y^*,z^*,x^*)Idis_2$ and $(z,x,y)Idis_1$ implies $(z^*,x^*,y^*)Idis_2$.  
Moreover $(y,z,x)Idis_1$ follows similarly as $(z,x,y)Idis_1$, so that we need to prove only $(x,y,z)Idis_1$ and $(z,x,y)Idis_1$. \\
Let us prove $(x,y,z)Idis_1$. 
Since by $(IOM^{'})$ and $(IOM)$, $x\le_Q x^*\ra y$ and $y\le_Q x^*\ra y$, we get 
$(x^*\ra y)\ra z^*\le_Q x\ra z^*$, $(x^*\ra y)\ra z^*\le_Q y\ra z^*$, so that 
$((x^*\ra y)\ra z^*)^*\ge_Q (x\ra z^*)^*$, $((x^*\ra y)\ra z^*)^*\ge_Q (y\ra z^*)^*$. 
Applying Lemma \ref{ioml-60-10}$(2)$ for $X:=((x^*\ra y)\ra z^*)^*$, $Y:=(x\ra z^*)^*$, $Z:=(y\ra z^*)^*$, we get 
$((x^*\ra y)\ra z^*)^*\ge_Q(x\ra z^*)\ra (y\ra z^*)^*$. \\
By Proposition \ref{ioml-40}$(2)$, $u\le_Q v$ implies $v=v\Cup u=u^*\ra (v\ra u)^*$. 
Denoting $u=(x\ra z^*)\ra (y\ra z^*)^*$, $v=((x^*\ra y)\ra z^*)^*$, we have: \\
$\hspace*{1.00cm}$ $v\ra u=((x^*\ra y)\ra z^*)^*\ra((x\ra z^*)\ra (y\ra z^*)^*)$ \\ 
$\hspace*{2.10cm}$ $=(x^*\ra y)\ra (z\ra ((x\ra z^*)\ra (y\ra z^*)^*))$ (by Lemma \ref{qbe-10}$(7))$ \\
$\hspace*{2.10cm}$ $=(x^*\ra y)\ra ((z\ra (x\ra z^*)^*)^*\ra (y\ra z^*)^*)$ (by Lemma \ref{qbe-10}$(7))$ \\
$\hspace*{2.10cm}$ $=(x^*\ra y)\ra ((z\ra (z\ra x^*)^*)^*\ra (y\ra z^*)^*)$ \\ 
$\hspace*{2.10cm}$ $=(x^*\ra y)\ra ((z\ra x)^*\ra (y\ra z^*)^*)$ (by Prop. \ref{dioml-60}, $z\mathcal{C} x^*$                                                                implies \\ $\hspace*{10.00cm}$ $z\ra (z\ra x^*)^*=z\ra x$) \\
$\hspace*{2.10cm}$ $=(x^*\ra y)\ra ((x^*\ra z^*)^*\ra (y\ra z^*)^*)$ \\
$\hspace*{2.10cm}$ $=(x^*\ra y)\ra (x^*\ra (z\ra (y\ra z^*)^*))$ (by Lemma \ref{qbe-10}$(7))$ \\
$\hspace*{2.10cm}$ $=(x^*\ra y)\ra (x^*\ra (z\ra (z\ra y^*)^*))$ \\
$\hspace*{2.10cm}$ $=(x^*\ra y)\ra (x^*\ra (z\ra y))$ (by Prop. \ref{dioml-60}, $z\mathcal{C} y^*$ 
                                            implies \\ $\hspace*{10.00cm}$ $z\ra (z\ra y^*)^*=z\ra y$) \\
$\hspace*{2.10cm}$ $=(x^*\ra y)\ra (z\ra (x^*\ra y))$ \\
$\hspace*{2.10cm}$ $=z\ra ((x^*\ra y)\ra (x^*\ra y))=z\ra 1=1$. \\ 
It follows that: \\
$\hspace*{2.00cm}$ $((x^*\ra y)\ra z^*)^*=v=u^*\ra (v\ra u)^*=((x\ra z^*)\ra (y\ra z^*)^*)^*\ra 0$ \\
$\hspace*{5.00cm}$ $=(x\ra z^*)\ra (y\ra z^*)^*$. \\
Hence $(x,y,z)Idis_1$. 
We prove now $(z,x,y)Idis_1$. \\
From $z\le_Q z^*\ra x$, $x\le_Q z^*\ra x$, we get $((z^*\ra x)\ra y^*)^*\ge_Q (z\ra y^*)^*$ and 
$((z^*\ra x)\ra y^*)^*\ge_Q (x\ra y^*)^*$, respectively. 
By Lemma \ref{ioml-60-10}$(2)$, we have $((z^*\ra x)\ra y^*)^*\ge_Q (z\ra y^*)\ra (x\ra y^*)^*$. 
By Proposition \ref{ioml-40}$(2)$, $u\le_Q v$ implies $v=v\Cup u=u^*\ra (v\ra u)^*$. 
Denoting $u=(z\ra y^*)\ra (x\ra y^*)^*$, $v=((z^*\ra x)\ra y^*)^*$, we get: \\
$\hspace*{1.00cm}$ $v\ra u=((z^*\ra x)\ra y^*)^*\ra ((z\ra y^*)\ra (x\ra y^*)^*)$ \\
$\hspace*{2.10cm}$ $=(z^*\ra x)\ra (y\ra ((z\ra y^*)\ra (x\ra y^*)^*))$ (by Lemma \ref{qbe-10}$(7))$ \\ 
$\hspace*{2.10cm}$ $=(z^*\ra x)\ra ((y\ra (z\ra y^*))^*\ra (x\ra y^*)^*)$ (by Lemma \ref{qbe-10}$(7))$ \\
$\hspace*{2.10cm}$ $=(z^*\ra x)\ra ((y\ra (y\ra z^*))^*\ra (x\ra y^*)^*)$ \\
$\hspace*{2.10cm}$ $=(z^*\ra x)\ra ((y\ra z)^*\ra (x\ra y^*)^*)$ (by Prop. \ref{dioml-60}, $y\mathcal{C} z^*$                                                                implies \\ $\hspace*{10.00cm}$ $y\ra (y\ra z^*)^*=y\ra z$) \\
$\hspace*{2.10cm}$ $=(z^*\ra x)\ra ((z^*\ra y^*)^*\ra (x\ra y^*)^*)$ \\
$\hspace*{2.10cm}$ $=(z^*\ra x)\ra (z^*\ra (y\ra (x\ra y^*)^*$ (by Lemma \ref{qbe-10}$(7))$ \\
$\hspace*{2.10cm}$ $=((z^*\ra x)\ra z)^*\ra (y\ra (x\ra y^*)^*)$ (by Lemma \ref{qbe-10}$(7))$ \\
$\hspace*{2.10cm}$ $=(z^*\ra x^*)^*\ra (y\ra (x\ra y^*)^*)$ (by Prop. \ref{dioml-60}, since $z^*\mathcal{C} x$, \\
                                      $\hspace*{8.00cm}$ $(z^*\ra x)\ra z=z^*\ra (z^*\ra x)^*=z^*\ra x^*$) \\
$\hspace*{2.10cm}$ $=z^*\ra (x\ra (y\ra (x\ra y^*)^*))$ (by Lemma \ref{qbe-10}$(7))$ \\
$\hspace*{2.10cm}$ $=z^*\ra ((x\ra y^*)^*\ra (x\ra y^*)^*)$ (by Lemma \ref{qbe-10}$(7))$ \\
$\hspace*{2.10cm}$ $=z^*\ra 1=1$. \\
It follows that: \\
$\hspace*{2.00cm}$ $((z^*\ra x)\ra y^*)^*=v=u^*\ra (v\ra u)^*=((z\ra y^*)\ra (x\ra y^*)^*)^*\ra 0$ \\
$\hspace*{5.00cm}$ $=(z\ra y^*)\ra (x\ra y^*)^*$. \\
Thus $(z,x,y)Idis_1$.
We conclude that $(x,y,z)$ is a distributive triple. 
\end{proof}

\begin{corollary} \label{dioml-100-10} Let $X$ be an implicative-orthomodular lattice and let $x_1,x_2,y\in X$, such that, for each $i=1,2$, $x_i\le_Q y$ or $x_i\ge_Q y$. Then the triple $(x_1,x_2,y)$ is distributive. 
\end{corollary}
\begin{proof}
If $x_i\le_Q y$ or $x_i\ge_Q y$, then $x_i\mathcal{C} y$, $i=1,2$, and we apply Theorem \ref{dioml-100}.
\end{proof}

The next result allows us to characterize the distributivity of an implicative-orthomodular lattice by 
divisibility property. 

\begin{proposition} \label{dioml-110} Let $X$ be an implicative involutive BE algebra, and let $x,y,z\in X$.  
If $(z,x)Idiv$ and $(z,y)Idiv$, then $(x,y,z)Idis_1$.  
\end{proposition}
\begin{proof}
Assume that $(z,x)Idiv$ and $(z,y)Idiv$. 
Since $x\le_Q x^*\ra y$ and $y\le_Q x^*\ra y$, we have 
$(x^*\ra y)\ra z^*\le_Q x\ra z^*$, $(x^*\ra y)\ra z^*\le_Q y\ra z^* $, so that 
$((x^*\ra y)\ra z^*)^*\ge_Q (x\ra z^*)^*$, $((x^*\ra y)\ra z^*)^*\ge_Q (y\ra z^*)^*$. 
Applying Lemma \ref{ioml-60-10}$(2)$ for $X:=((x^*\ra y)\ra z^*)^*$, $Y:=(x\ra z^*)^*$, $Z:=(y\ra z^*)^*$, we get 
$((x^*\ra y)\ra z^*)^*\ge_Q(x\ra z^*)\ra (y\ra z^*)^*$. \\
Now we prove that $((x^*\ra y)\ra z^*)^*\le (x\ra z^*)\ra (y\ra z^*)^*$. Indeed, we have: \\ 
$\hspace*{1.80cm}$ $x^*\ra y\le_Q z\ra (x^*\ra y)$ \\
$\hspace*{3.00cm}$ $=(z\ra z^*)^*\ra (x^*\ra y)$ (by $(iG)$) \\
$\hspace*{3.00cm}$ $=z\ra (z\ra (x^*\ra y))$ (by Lemma \ref{qbe-10}$(7))$ \\
$\hspace*{3.00cm}$ $=z\ra (x^*\ra (z\ra y))$ \\
$\hspace*{3.00cm}$ $=(z\ra x)^*\ra (z\ra y)$ (by Lemma \ref{qbe-10}$(7))$ \\
$\hspace*{3.00cm}$ $=(z\ra (z\ra x^*)^*)^*\ra (z\ra (z\ra y^*)^*)$ (since $(z,x)Idiv$, $(z,y)Idiv$) \\
$\hspace*{3.00cm}$ $=z\ra ((z\ra x^*)\ra (z\ra (z\ra y^*)^*))$ (by Lemma \ref{qbe-10}$(7))$ \\
$\hspace*{3.00cm}$ $=z\ra (z\ra ((z\ra x^*)\ra (z\ra y^*)^*))$ \\
$\hspace*{3.00cm}$ $=(z\ra z^*)^*\ra ((z\ra x^*)\ra (z\ra y^*)^*)$ (by Lemma \ref{qbe-10}$(7))$ \\
$\hspace*{3.00cm}$ $=z\ra ((z\ra x^*)\ra (z\ra y^*)^*)$ (by $(iG)$). \\
Hence $x^*\ra y\le_Q z\ra ((z\ra x^*)\ra (z\ra y^*)^*)$, so that, by Proposition \ref{qbe-20}$(4)$,  
$(x^*\ra y)\ra (z\ra ((z\ra x^*)\ra (z\ra y^*)^*))=1$. 
Applying Lemma \ref{qbe-10}$(7))$, it follows that $((x^*\ra y)\ra z^*)^*\ra ((z\ra x^*)\ra (z\ra y^*)^*)=1$, 
that is $((x^*\ra y)\ra z^*)^*\le ((z\ra x^*)\ra (z\ra y^*)^*)$. 
Taking into consideration that $((x^*\ra y)\ra z^*)^*\ge_Q(x\ra z^*)\ra (y\ra z^*)^*$, by  
Proposition \ref{ioml-50}$(6)$, we get $((x^*\ra y)\ra z^*)^*=(x\ra z^*)\ra (y\ra z^*)^*$, that is $(x,y,z)Idis_1$. 
\end{proof}

\begin{theorem} \label{dioml-120} An implicative-orthomodular lattice $X$ is distributive if and only if it    
satisfies condition $(Idiv)$.    
\end{theorem}
\begin{proof}
Assume that $X$ satisfies condition $(Idiv)$. Since $(x,y)Idiv$, for each $x,y\in X$, then by Proposition \ref{dioml-60}, $x\mathcal{C} y$, for each $x,y\in X$. Applying Proposition \ref{dioml-110} and Lemma \ref{dioml-30}, 
it follows that $(x,y,z)Idis_1$ and $(x^*,y^*,z^*)Idis_1$, 
for each $x,y,z\in X$. Since $(x^*,y^*,z^*)Idis_1$ is equivalent to $(x,y,z)Idis_2$, it follows that 
$(x,y,z)Idis$, for each $x,y,z\in X$, hence $X$ is distributive. \\
Conversely, assume that $(x,y,z)Idis_1$ holds, thus $((x^*\ra y)\ra z^*)^*=(x\ra z^*)\ra (y\ra z^*)^*$, 
for all $x,y,z\in X$. 
Replacing $y$ by $z^*$ in $(x,y,z)Idis_1$ we get $z\ra (z\ra x)^*=z\ra x^*$, so that $(z,x)Idiv$, for all $z,x\in X$,  
hence $X$ satisfies condition $(Idiv)$. 
\end{proof}

$\vspace*{1mm}$

\section{Implicative-Boolean center of implicative-orthomodular lattices} 

In this section, we prove in the case of implicative-orthomodular lattices other notions and results studied in \cite{Holl1} for orthomodular lattices. 
We define the center $\mathcal{C}(X)$ of an implicative-orthomodular lattice $X$ as the 
set of all elements of $X$ that commute with all other elements of $X$, and we prove that $\mathcal{C}(X)$ is an implicative-Boolean subalgebra of $X$. 
We give a characterization of $\mathcal{C}(X)$ with respect to the set $\mathcal{K}(x)$ of all complements of an 
element $x\in X$. 
As a generalization of the concept of the center of an implicative-orthomodular lattice $X$, we define the 
commutor $Y^c$ of a subset $Y$ of $X$ proving that $Y^c$ is a sublattice of $X$ containing $\mathcal{C}(X)$. 

\begin{proposition} \label{dioml-130} 
Let $X$ be an implicative-orthomodular lattice and let $x,y,z\in X$. If $x\mathcal{C} z$ and $y\mathcal{C} z$, 
then $(x\ra y) \mathcal{C} z$. 
\end{proposition}
\begin{proof}
By $(IOM^{'})$ and $(IOM)$, we have $x^*\le_Q x\ra y$ and $y\le_Q x\ra y$, respectively. 
Applying Lemma \ref{ioml-60-10}$(1)$, we get: \\ 
$\hspace*{2cm}$ $(x^*\ra z^*)^*\le_Q ((x\ra y)\ra z^*)^*$ and $(y\ra z^*)^*\le_Q ((x\ra y)\ra z^*)^*$. \\
Using Lemma \ref{ioml-60-10}$(2)$, it follows that: \\
$\hspace*{2cm}$ $(x^*\ra z^*)\ra (y\ra z^*)^*\le_Q ((x\ra y)\ra z^*)^*$. \\
Similarly we have: \\
$\hspace*{2cm}$ $(x^*\ra z)\ra (y\ra z)^*\le_Q ((x\ra y)\ra z)^*$. \\
By Lemma \ref{ioml-60-10}$(3)$, we get: \\
$\hspace*{2cm}$ $((x^*\ra z^*)\ra (y\ra z^*)^*)^*\ra ((x^*\ra z)\ra (y\ra z)^*)\le_Q$ \\ 
$\hspace*{5cm}$ $((x\ra y)\ra z^*)\ra ((x\ra y)\ra z)^*$, \\
and by Lemma \ref{qbe-10-10}, it follows that: \\    
$\hspace*{2cm}$ $((x^*\ra z^*)\ra (x^*\ra z)^*)^*\ra ((y\ra z^*)\ra (y\ra z)^*)\le_Q$ \\ 
$\hspace*{5cm}$ $((x\ra y)\ra z^*)\ra ((x\ra y)\ra z)^*$. \\
Since $x^*\mathcal{C} z$ and $y\mathcal{C} z$, we get: \\
$\hspace*{2cm}$ $x\ra y\le_Q ((x\ra y)\ra z^*)\ra ((x\ra y)\ra z)^*$. \\
On the other hand, by Lemma \ref{ioml-60-10}$(4)$, \\
$\hspace*{2cm}$ $((x\ra y)\ra z^*)\ra ((x\ra y)\ra z)^*\le_Q x\ra y$. \\
Since $\le_Q$ is antisymmetric, it follows that \\
$\hspace*{2cm}$ $((x\ra y)\ra z^*)\ra ((x\ra y)\ra z)^*= x\ra y$, \\
hence $(x\ra y) \mathcal{C} z$. 
\end{proof}

\begin{definition} \label{dioml-140}  
\emph{
The \emph{center} of $X$ is the set $\mathcal{C}(X)=\{x\in X\mid x \mathcal{C} y$, for all $y\in X\}$.  
}
\end{definition}

\begin{remark} \label{dioml-140-10} Let $X$ be an implicative-orthomodular lattice and let $x_1,x_2,y\in X$, 
such that $x_1,x_2\in \mathcal{C}(X)$ or $y\in \mathcal{C}(X)$. Then the triple $(x_1,x_2,y)$ is distributive. 
\end{remark}
\begin{proof}
If $x_1,x_2\in \mathcal{C}(X)$ or $y\in \mathcal{C}(X)$, then $x_1\mathcal{C} y$ and $x_2\mathcal{C} y$, 
and we apply Theorem \ref{dioml-100}.
\end{proof}

The notion of \emph{implicative-Boolean algebras} was introduced by A. Iorgulescu in 2009 as 
implicative involutive BE algebras $(X,\ra,^*,1)$ satisfying condition $(Idiv)$ (see \cite[Def. 3.4.12]{Ior35}). 

\begin{theorem} \label{dioml-150} The center $\mathcal{C}(X)$ of an implicative-orthomodular lattice $X$ is an 
implicative-Boolean subalgebra of $X$. 
\end{theorem}
\begin{proof}
According to Lemma \ref{dioml-20}, $0,1\in \mathcal{C}(X)$, and, by Proposition \ref{dioml-130}, 
$x\ra y\in \mathcal{C}(X)$, whenever $x,y\in \mathcal{C}(X)$. 
Hence $\mathcal{C}(X)$ is a subalgebra of $X$. 
By Foulis-Holland theorem, $\mathcal{C}(X)$ satisfies condition $(Idis)$, and according to Theorem  \ref{dioml-120}, it also satisfies condition $(Idiv)$. 
It follows that $\mathcal{C}(X)$ is an implicative-Boolean algebra. 
\end{proof}

\begin{corollary} \label{dioml-150-10} If $X$ in an implicative-orthomodular lattice such that $x\mathcal{C} y$ 
for all $x,y\in X$, then $X$ is an implicative-Boolean algebra. 
\end{corollary}

For $x\in X$, denote $\mathcal{K}(x)=\{z\in X\mid x\ra z^*=x^*\ra z=1\}$, that is the set of all complements of $x$. 
Obviously $x^*\in \mathcal{K}(x)$. 

\begin{theorem} \label{dioml-160} Let $X$ be an implicative-orthomodular lattice, let $x\in X$ and denote  $x_z=(z\ra (z\ra x^*)^*)\ra (z^*\ra x)^*$, with $z\in X$. The following hold: \\
$(1)$ $x_z\in \mathcal{K}(x);$ \\
$(2)$ for any $y\in \mathcal{K}(x)$, there exists $z\in X$ such that $y=x_z;$ \\
$(3)$ $x_z=x^*$ if and only if $x\mathcal{C} z$. 
\end{theorem}
\begin{proof}
$(1)$ We have: \\
$\hspace*{2.00cm}$ $x^*\ra x_z=x^*\ra ((z\ra (z\ra x^*)^*)\ra (z^*\ra x)^*)$ \\
$\hspace*{3.50cm}$ $=(z\ra (z\ra x^*)^*)\ra (x^*\ra (z^*\ra x)^*)$ \\
$\hspace*{3.50cm}$ $=((x\ra z^*)\ra z^*)\ra ((z^*\ra x)\ra x)$ \\
$\hspace*{3.50cm}$ $=(x\Cup z^*)\ra (z^*\Cup x)$ \\
$\hspace*{3.50cm}$ $=1$ (by Prop. \ref{ioml-70}). \\
By Lemma \ref{ioml-60-10}$(5)$, we have $(z\ra (z\ra x^*)^*)^*\le_Q z^*\ra x$, so that 
$(z\ra (z\ra x^*)^*)^*\mathcal{C} (z^*\ra x)$, hence $(z^*\ra x)^*\mathcal{C} (z\ra (z\ra x^*)^*)^*$. 
Moreover, since $x\le_Q z^*\ra x$, we get $x\mathcal{C} (z^*\ra x)$, hence $(z^*\ra x)^*\mathcal {C} x$. 
By Foulis-Holland theorem, it follows that $((z\ra (z\ra x^*)^*)^*, (z^*\ra x)^*, x)Idis_1$. 
Denoting $X:=(z\ra (z\ra x^*)^*)^*$, $Y:=(z^*\ra x)^*$, $Z:=x$, we have: \\
$\hspace*{2.00cm}$ $(X^*\ra Y)\ra Z^*=((z\ra (z\ra x^*)^*)\ra (z^*\ra x)^*)\ra x^*$ \\
$\hspace*{4.90cm}$ $=x\ra ((z\ra (z\ra x^*)^*)\ra (z^*\ra x)^*)^*=x\ra x_z^*$, and \\
$\hspace*{2.00cm}$ $(X\ra Z^*)\ra (Y\ra Z^*)^*=(Z\ra X^*)\ra (Z\ra Y^*)^*$ \\
$\hspace*{6.20cm}$ $=(x\ra (z\ra (z\ra x^*)^*))\ra (x\ra (z^*\ra x))^*$ \\
$\hspace*{6.20cm}$ $=(x\ra (z\ra (z\ra x^*)^*)\ra 1^*$ \\
$\hspace*{6.20cm}$ $=(x\ra ((z\ra x^*)\ra z^*))^*$ \\
$\hspace*{6.20cm}$ $=((x\ra z^*)\ra (x\ra z^*))^*=1^*=0$. \\
Since $((X^*\ra Y)\ra Z^*)^*=(X\ra Z^*)\ra (Y\ra Z^*)^*$, it follows that $(x\ra x_z^*)^*=0$, that is $x\ra x_z^*=1$. 
Hence $x_z\in \mathcal{K}(x)$. \\ 
$(2)$ Let $y\in \mathcal{K}(x)$, that is $x^*\ra y=x\ra y^*=1$. Taking $z:=y$, we get: \\
$\hspace*{2.00cm}$ $x_z=x_y=(y\ra (y\ra x^*)^*)\ra (y^*\ra x)^*$ \\
$\hspace*{3.50cm}$ $=(y\ra (x\ra y^*)^*)\ra (x^*\ra y)^*$ \\
$\hspace*{3.50cm}$ $=(y\ra 1^*)\ra 1^*=y^{**}=y$. \\
$(3)$ If $z\mathcal{C} x$, then $z\mathcal{C} x^*$ and $x^*\mathcal{C} z$. 
According to Proposition \ref{dioml-60}, $z\mathcal{C} x^*$ implies $(z,x^*)Idiv$, and we have: \\
$\hspace*{2.00cm}$ $x_z=(z\ra (z\ra x^*)^*)\ra (z^*\ra x)^*$ \\
$\hspace*{2.50cm}$ $=((z\ra x)\ra (z^*\ra x)^*$ (since $(x^*,z)Idiv$) \\
$\hspace*{2.50cm}$ $=(x^*\ra z^*)\ra (x^*\ra z)^*$ \\
$\hspace*{2.50cm}$ $=x^*$ (since $x^*\mathcal{C} z$). \\
Conversely, suppose $x_z=x^*$, and we have successively: \\
$\hspace*{2.00cm}$ $(z\ra (z\ra x^*)^*)\ra (z^*\ra x)^*=x^*$, \\
$\hspace*{2.00cm}$ $(z^*\ra x)\ra (z\ra (z\ra x^*)^*)^*=x^*$, \\
$\hspace*{2.00cm}$ $(z\ra (z\ra x^*)^*)^*\le_Q (z^*\ra x)\ra (z\ra (z\ra x^*)^*)^*=x^*$, \\
$\hspace*{2.00cm}$ $(z\ra (z\ra x^*)^*)^*\le_Q x^*$. \\
On the other hand, $z^*\le_Q z\ra (z\ra x^*)^*$, so that $(z\ra (z\ra x^*)^*)^*\le_Q z$. 
Applying Lemma \ref{qbe-20-10}, we get $(z\ra (z\ra x^*)^*)^*\le_Q (x^*\ra z^*)^*=(z\ra x)^*$. 
It follows that $z\ra x\le_Q z\ra (z\ra x^*)^*$. 
Since $x^*\le_Q z\ra x^*$ implies $(z\ra x^*)^*\le_Q x$, we have $z\ra (z\ra x^*)^*\le_Q z\ra x$, hence 
$z\ra (z\ra x^*)^*=z\ra x$. 
It follows that $(z,x^*)Idiv$, and applying Proposition \ref{dioml-60}, we get $z\mathcal{C} x^*$, and finally, 
$x\mathcal{C} z$. 
\end{proof}

\begin{corollary} \label{dioml-170} Let $X$ be an implicative-orthomodular lattice and let $x\in X$. 
The following are equivalent: \\  
$(a)$ $x\in \mathcal{C}(X);$ \\
$(b)$ $\mathcal{K}(x)=\{x^*\}$. 
\end{corollary}
\begin{proof}
By Theorem \ref{dioml-160}$(1)$,$(2)$, it follows that $\mathcal{K}(x)=\{x_z\mid z\in X\}$.
Since by Theorem \ref{dioml-160}$(3)$, $x_z=x^*$ if and only if $x\mathcal{C} z$, we have $\mathcal{K}(x)=\{x^*\}$ 
if and only if $x\mathcal{C} z$, for all $z\in X$.  
Hence $\mathcal{K}(x)=\{x^*\}$ if and only if $x\in \mathcal{C}(X)$. 
\end{proof}

The notion of a commutor generalizes the concept of the center of an implicative-orthomodular lattice. 

\begin{definition} \label{dioml-180} \emph{
Let $X$ be an implicative-ortomodular lattice, and let $Y\subseteq X$ be a nonempty subset of $X$. 
The \emph{commutor} of $Y$ is the set $Y^c=\{x\in X\mid x\mathcal{C} y $, for all $y\in Y\}$. 
}
\end{definition}

\begin{proposition} \label{dioml-190}
Let $X$ be an implicative-ortomodular lattice, and let $Y\subseteq X$ be a nonempty subset of $X$. 
The following hold: \\
$(1)$ $Y^c$ is a sublattice of $X;$ \\
$(2)$ $\mathcal{C}(X)\subseteq Y^c;$ \\
$(3)$ $X^c=X$. 
\end{proposition}
\begin{proof}
$(1)$ By Lemma \ref{dioml-20}, $0,1\in Y^c$. Let $x,y\in Y^c$, that is $x\mathcal{C} z$ and $y\mathcal{C} z$ for 
all $z\in Y$. By Proposition \ref{dioml-130}, $(x\ra y)\mathcal{C}z$, hence $x\ra y\in Y^c$. 
Hence $Y^c$ is a sublattice of $X$. \\
$(2)$ and $(3)$ are obvious. 
\end{proof}

$\vspace*{1mm}$

\section{Applications of Foulis-Holland theorem} 

Foulis-Holland theorem provides a method to prove new properties of implicative-orthomodular lattices. 
We prove for the case of implicative-orthomodular lattices a characterization of orthomodular 
lattices given in \cite{Bonzio1} by S. Bonzio and I. Chajda. 
We also prove new properties, and we give new proofs of certain known properties of implicative-orthomodular lattices. 

\begin{theorem} \label{dioml-120-10} Let $X$ be an implicative involutive BE algebra. Then following are equivalent: \\
$(a)$ $X$ is an implicative-orthomodular lattice; \\
$(b)$ for all $x,y\in X$, \\
$\hspace*{3.00cm}$ $((x^*\ra y^*)\ra (x^*\ra y)^*)^*=x^*\ra (x\Cap y^*)$.    
\end{theorem} 
\begin{proof} Let $X$ be an implicative-orthomodlar lattice, and let $x,y\in X$. By $(IOM)$, $(IOM^{'})$, we have 
$x\le_Q x^*\ra y$ and $y\le_Q x^*\ra y$. 
Using Lemma \ref{dioml-20}, we get $x\mathcal{C} (x^*\ra y)$, $y\mathcal{C} (x^*\ra y)$, so that, by 
Theorem \ref{dioml-10-10} and Lemma \ref{dioml-30}, it follows that $(x^*\ra y)\mathcal{C} x$ and 
$(x^*\ra y)\mathcal{C} y^*$. 
Hence, by Foulis-Holland theorem, the triple $(x,y^*,x^*\ra y)$ is distributive, that is it satisfies 
condition $(Idis_1)$. It follows that: \\
$\hspace*{2.00cm}$ $((x^*\ra y^*)\ra (x^*\ra y)^*)^*=(x\ra (x^*\ra y)^*)\ra (y^*\ra (x^*\ra y)^*)^*$ \\
$\hspace*{6.50cm}$ $=((x^*\ra y)\ra x^*)\ra ((x^*\ra y)\ra y)^*$ \\
$\hspace*{6.50cm}$ $=x^*\ra (x^*\Cup y)^*$ (by $(Impl)$) \\
$\hspace*{6.50cm}$ $=x^*\ra (x\Cap y^*)$. \\
Conversely, assume that $X$ satisfies the above identity, and let $x,y\in X$ such that $x\le_L y$. 
Then also $y^*\le_L x^*$, so that $y=x^*\ra y$, and we have: 
$((x^*\ra y^*)\ra (x^*\ra y))^*=((x^*\ra y^*)\ra y^*)^*=x\Cap y$. 
Moreover, since $y^*\le_L x^*$ implies $y=x^*\ra y$, we get: \\
$\hspace*{2.00cm}$ $x^*\ra (x\Cap y^*)=x^*\ra (x^*\Cup y)^*=(x^*\Cup y)\ra x$ \\ 
$\hspace*{4.40cm}$ $=((x^*\ra y)\ra y)\ra x=(y\ra y)\ra x=x$. \\
It follows that $x=x\Cap y$, so that $x\le_Q y$. Hence $x\le_L y$ implies $x\le_Q y$, and according to 
Theorem \ref{dioml-05-50}, $X$ is an implicative-orthomodular lattice.   
\end{proof}

\begin{remark} \label{dioml-120-20} 
Using the mutually inverse transformations \\
$\hspace*{3cm}$ $\varphi:$\hspace*{0.2cm}$ x\ra y:=(x\wedge y^{'})^{'}$ $\hspace*{0.1cm}$ and  
                $\hspace*{0.1cm}$ $\psi:$\hspace*{0.2cm}$ x\wedge y:=(x\ra y^*)^*$, \\
and the relation $x\vee y:=(x^{'}\wedge y^{'})^{'}=x^*\ra y$, we can easily check that the identity from 
Theorem \ref{dioml-120-10} is equivalent to the identity \\
$\hspace*{3.00cm}$ $x\vee (y^{'}\wedge (x\vee y))=(x\vee y^{'})\wedge (x\vee y)$, \\
used in \cite[Th. 2]{Bonzio1} in the case of an ortholattice $(L,\vee,\wedge,^{'},0,1)$.  
\end{remark}

The next result gives a necessary condition for an implicative involutive BE algebra to be an 
implicative-orhomodular lattice.  
\begin{proposition} \label{dioml-120-30} Let $X$ be an implicative-orhomodular lattice. Then the following holds, 
for all $x,y\in X:$ \\
$\hspace*{3.00cm}$ $x^*\ra (x\Cap y^*)=x^*\ra (x\Cap y)$.    
\end{proposition} 
\begin{proof} Let $x,y\in X$. 
Similarly as in Theorem \ref{dioml-120-10}, we get that the triple $(x,y,x^*\ra y^*)$ is distributive, so that 
it satisfies condition $(Idis_1)$. Hence: \\
$\hspace*{2.00cm}$ $((x^*\ra y)\ra (x^*\ra y^*)^*)^*=(x\ra (x^*\ra y^*)^*)\ra (y\ra (x^*\ra y^*)^*)^*$ \\
$\hspace*{6.50cm}$ $=((x^*\ra y^*)\ra x^*)\ra ((x^*\ra y^*)\ra y^*)^*$ \\
$\hspace*{6.50cm}$ $=x^*\ra (x\Cap y)$ (by $(Impl)$). \\
On the other hand, $((x^*\ra y)\ra (x^*\ra y^*)^*)^*=((x^*\ra y^*)\ra (x^*\ra y)^*)^*$, so that, by Theorem \ref{dioml-120-10}, 
we get $x^*\ra (x\Cap y^*)=x^*\ra (x\Cap y)$. 
\end{proof}

\begin{example} \label{dioml-120-40} 
Consider the algebra $\mathbf{O}_6=(O_6,\ra,^*,1)$, with $O_6=\{0,x,y,x^*,y^*,1\}$ and $\ra$ defined below. 
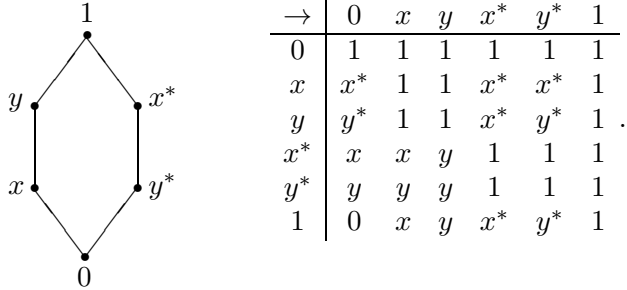
\[
\begin{picture}(50,-70)(0,60)
\put(37,11){\circle*{3}}
\put(34,0){$0$}
\put(37,11){\line(3,4){20}}
\put(57,37){\circle*{3}}
\put(61,35){$y^*$}
\put(37,11){\line(-3,4){20}}
\put(18,37){\circle*{3}}
\put(8,35){$x$}
\put(18,37){\line(0,1){30}}
\put(18,68){\circle*{3}}
\put(8,68){$y$}
\put(57,37){\line(0,1){30}}
\put(57,68){\circle*{3}}
\put(61,68){$x^*$}
\put(18,68){\line(3,4){20}}
\put(38,95){\circle*{3}}
\put(35,100){$1$}
\put(57,68){\line(-3,4){20}}
\end{picture}
\hspace*{2cm}
\begin{array}{c|cccccc}
\rightarrow & 0 & x & y & x^* & y^* & 1 \\ \hline
0   & 1   & 1 & 1 & 1   & 1   & 1 \\
x   & x^* & 1 & 1 & x^* & x^* & 1 \\
y   & y^* & 1 & 1 & x^* & y^* & 1 \\
x^* & x   & x & y & 1   & 1   & 1 \\
y^* & y   & y & y & 1   & 1   & 1 \\
1   & 0   & x & y & x^* & y^* & 1
\end{array}
.
\]

$\vspace*{5mm}$

We can easily check that $\mathbb{O}_6$ is an implicative involutive BE algebra, but it is not an 
implicative-orthomodular lattice. 
Indeed, we have $(x\ra y^*)^*=x$ and $x\Cap y=y$, so that $x\le_L y$, but $x\nleq_Q y$, hence, by 
Theorem \ref{dioml-05-50}, $\mathbb{O}_6$ is not an implicative-orthomodular lattice.  
\end{example} 

\begin{remark} \label{dioml-120-50} An implicative-orthomodular lattice $X$ does not contain a subalgebra 
isomorphic to $\mathbb{O}_6$. 
Indeed, in such a case, $x=x^*\ra (x\Cap y^*)=x^*\ra (x\Cap y)=y$, a contradiction, so that the identity from Proposition \ref{dioml-120-30} is not satisfied. 
\end{remark}

\begin{proposition} \label{dioml-120-60} Let $X$ be an implicative-orthomodular lattice, and let $x,y\in X$. 
Then $x\ra (x^*\Cap y^*)=x^*$ iff $y\ra (y^*\Cap x^*)=y^*$. 
\end{proposition}
\begin{proof}
Since $x^*\le_Q x\ra y$ and $y\le_Q x\ra y$, we get $x^*\mathcal{C} (x\ra y)$ and $y\mathcal{C} (x\ra y)$, so that 
$(x\ra y)\mathcal{C} x^*$ and $(x\ra y)\mathcal{C} y^*$. By Foulis-Holland theorem, the triple $(x^*,y^*,x\ra y)$ 
is distributive, and applying $(Idis_1)$ for $X:=x^*$, $Y:=y^*$ and $Z:=x\ra y$, we have: \\
$\hspace*{2.00cm}$ $((x\ra y^*)\ra (x\ra y)^*)^*=(x^*\ra (x\ra y)^*)\ra (y^*\ra (x\ra y)^*)^*$ \\
$\hspace*{6.15cm}$ $=((x\ra y)\ra x)\ra ((x\ra y)\ra y)^*$ \\
$\hspace*{6.15cm}$ $=x\ra (x\Cup y)^*$ (by $(Impl)$) \\
$\hspace*{6.15cm}$ $=x\ra (x^*\Cap y^*)$. \\
Similarly, the triple $(y^*,x^*,y\ra x)$ is distributive, and we get: 
$((y\ra x^*)\ra (y\ra x)^*)^*=y\ra (y^*\Cap x^*)$. 
If $x\ra (x^*\Cap y^*)=x^*$, it follows that $((x\ra y^*)\ra (x\ra y)^*)^*=x^*$, hence $(x\ra y^*)\ra (x\ra y)^*=x$, 
that is $x\mathcal{C} y$. By Theorem \ref{dioml-10-10}, we have $y\mathcal{C} x$, so that  
$(y\ra x^*)\ra (y\ra x)^*=y$, hence $y\ra (y^*\Cap x^*)=((y\ra x^*)\ra (y\ra x)^*)^*=y^*$. 
The converse follows similarly. 
\end{proof}

Using the Foulis-Holland theorem, we can give new proofs of certain known properties of implicative-orthomodular lattices. 

\begin{proposition} \label{dioml-120-70} Let $X$ be an implicative-orthomodular lattice. The following hold 
for all $x,y\in X$: \\
$(1)$ $x\ra (y\Cap x)=x\ra y;$ \\
$(2)$ $(x\ra y)\ra (y\Cap x)=x;$ \\
$(3)$ $(x\ra y)\ra (y\ra x)=y\ra x;$ \\
$(4)$ $(x\Cup y)\ra (x\ra y)^*=y^*$.  
\end{proposition}
\begin{proof}
$(1)$ Since $x\mathcal{C} x^*$, and $x^*\mathcal{C} (y^*\ra x^*)$ implies $x\mathcal{C} (y^*\ra x^*)^*$, by Foulis-Holland theorem the triple 
$((y^*\ra x^*)^*,x^*,x)$ is distributive. Applying $(Idis_1)$ for $X:=(y^*\ra x^*)^*$, $Y:=x^*$, $Z:=x$, 
we get successively: \\
$\hspace*{2.00cm}$ $(((y^*\ra x^*)\ra x^*)\ra x^*)^*=((y^*\ra x^*)^*\ra x^*)\ra (x^*\ra x^*)^*$, \\
$\hspace*{2.00cm}$ $((y^*\Cup x^*)\ra x^*)^*=(x\ra (y^*\ra x^*))\ra 1^*$, \\
$\hspace*{2.00cm}$ $(x\ra (y\Cap x))^*=(x\ra (x\ra y))^*$, \\
$\hspace*{2.00cm}$ $x\ra (y\Cap y)=x\ra y$ (by $(Pimpl)$). \\
$(2)$ Obviously, $(x\ra y)\mathcal{C} (x\ra y)^*$ implies $(x\ra y)\mathcal{C} (y^*\ra x^*)^*$, and 
$x^*\mathcal{C} (x\ra y)$ implies $(x\ra y)\mathcal{C} x^*$. Hence the triple  $((y^*\ra x^*)^*,x^*,x\ra y)$ is distributive. By $(Idis_1)$ for $X:=(y^*\ra x^*)^*$, $Y:=x^*$, $Z:=x\ra y$, we have successively: \\ 
$\hspace*{1.00cm}$ $(((y^*\ra x^*)\ra x^*)\ra (x\ra y)^*)^*=((y^*\ra x^*)^*\ra (x\ra y)^*)\ra (x^*\ra (x\ra y)^*)^*$, \\
$\hspace*{1.00cm}$ $((y^*\Cup x^*)\ra (x\ra y)^*)^*=((x\ra y)^*\ra (x\ra y)^*)\ra ((x\ra y)\ra x)^*$, \\
$\hspace*{1.00cm}$ $((x\ra y)\ra (y\Cap x))^*=1\ra x^*$ (by $(Impl)$), \\
$\hspace*{1.00cm}$ $(x\ra y)\ra (y\Cap x)=x$.  \\
$(3)$ We have $((x\ra y)\ra (y\ra x))^*=((x\ra y)\ra (x^*\ra y^*))^*$, and using $(Idis_1)$ for the distributive 
triple $(x^*,y,(x^*\ra y^*)^*)$, it follows that: \\
$\hspace*{2.00cm}$ $((x\ra y)\ra (x^*\ra y^*))^*=(x^*\ra (x^*\ra y^*))\ra (y\ra (x^*\ra y^*))^*$ \\
$\hspace*{6.10cm}$ $=(x^*\ra y^*)\ra (x^*\ra (y\ra y^*))^*$ (by $(Pimpl)$) \\
$\hspace*{6.10cm}$ $=(x^*\ra y^*)\ra (x^*\ra y^*)^*$ (by $(iG)$) \\
$\hspace*{6.10cm}$ $=(x^*\ra y^*)^*$ (by $(iG)$) \\
$\hspace*{6.10cm}$ $=(y\ra x)^*$. \\
Hence $(x\ra y)\ra (y\ra x)=y\ra x$. \\
$(4)$ The triple $((x\ra y)^*,y,(x\ra y))$ is distributive, hence, by $(Idis_1)$ for $X:=(x\ra y)^*$, $Y:=y$, 
$Z:=x\ra y$, we have: \\
$\hspace*{1.75cm}$ $((x\Cup y)\ra (x\ra y)^*)^*=(((x\ra y)\ra y)\ra (x\ra y)^*))^*$ \\
$\hspace*{5.50cm}$ $=(((x\ra y)^*\ra (x\ra y)^*)\ra (y\ra (x\ra y)^*))^*$ \\
$\hspace*{5.50cm}$ $=1\ra ((x\ra y)\ra y^*)^*$ \\
$\hspace*{5.50cm}$ $=((y^*\ra x^*)\ra y^*)^*$ \\
$\hspace*{5.50cm}$ $=(y^*)^*=y$ (by $(Impl)$). \\
It follows that $(x\Cup y)\ra (x\ra y)^*=y^*$. 
\end{proof}

$\vspace*{5mm}$


$\vspace*{1mm}$

\end{document}